\documentclass[11pt]{amsart}
\usepackage[T1]{fontenc}
\usepackage[utf8]{inputenc}
\usepackage{amsmath,amssymb,amsthm}
\usepackage{listings}
\usepackage{booktabs}
\usepackage{hyperref}

\lstdefinestyle{python}{
  language=Python,
  basicstyle=\ttfamily\scriptsize,
  numbers=none,
  breaklines=true,
  showstringspaces=false,
  columns=fullflexible,
  frame=single,
  upquote=true,
}

\theoremstyle{plain}
\newtheorem{theorem}{Theorem}
\newtheorem{lemma}[theorem]{Lemma}
\newtheorem{proposition}[theorem]{Proposition}
\newtheorem{corollary}[theorem]{Corollary}
\theoremstyle{definition}
\newtheorem{remark}[theorem]{Remark}

\DeclareMathOperator{\fix}{fix}

\newcommand{\Tn}{T_n}
\newcommand{\Un}{U_n}
\newcommand{\Dn}{D_n}
\newcommand{\Vn}{V_n}
\newcommand{\up}{u}
\newcommand{\dn}{d}
\newcommand{\Hs}{H^{\mathrm{s}}_n}
\newcommand{\Hp}{H^{\mathrm{p}}_n}
\newcommand{\ind}[1]{\left[#1\right]}

\title[Non-adjacent cell triples in a triangular grid]
      {Non-adjacent triples of cells in a triangular grid:
       proofs of three conjectured generating functions}

\author{Tong Niu}

\subjclass[2020]{Primary 05A15; Secondary 05C30, 05B45}
\keywords{triangular grid, independent sets, Burnside's lemma,
quasi-polynomial, generating function, OEIS}

\date{\today}

\begin{document}

\begin{abstract}
Let $\Tn$ be the triangular grid of side $n$, with its $n^2$ unit cells, and let
two cells be \emph{side-adjacent} if they share an edge and \emph{point-adjacent}
if they share a vertex. We count the orbits of $3$-element sets of pairwise
non-adjacent cells under the rotation group $C_3$ and under the full symmetry
group $D_3$ of $\Tn$. Three of the four resulting sequences are
OEIS \texttt{A397730}, \texttt{A397731} and \texttt{A397732}, for each of which
the only formula on record is a generating function conjectured by Sean A.\
Irvine on 14 July 2026. We prove all three conjectures. The route is Burnside's
lemma, together with exact lattice counts for each fixed-point term. The
side-adjacency graph turns out to be bipartite, so its triangle count vanishes,
and the triangles of the point-adjacency graph have a clean structural
description, which gives that count too. Along the way we obtain the first
explicit closed forms for the three sequences, each a quasi-polynomial of degree
$6$ and period $3$, $6$, $3$ respectively. The same machinery settles the fourth
member of the family, whose closed form has been recorded at \texttt{A289230}
without proof since 2017, and the labeled count \texttt{A289224} that feeds it.
\end{abstract}

\maketitle

\section{Introduction}

\subsection{The grid and the two adjacency relations}

Fix $n \ge 1$ and let $\Tn$ be the equilateral triangular grid of side $n$: the
triangle of side $n$ cut by the three families of lattice lines into $n^2$ unit
cells, $\binom{n+1}{2}$ of them pointing the same way as $\Tn$ (we call these
\emph{up-cells}) and $\binom{n}{2}$ pointing the other way (\emph{down-cells}).

Oblique lattice coordinates are used throughout, they turn every count below
into a lattice-point count in a triangle. Put
\begin{align*}
  \Vn &= \{(x,y) \in \mathbb{Z}^2_{\ge 0} : x + y \le n\},\\
  \Un &= \{(i,j) \in \mathbb{Z}^2_{\ge 0} : i + j \le n-1\}, \qquad
  \Dn = \{(i,j) \in \mathbb{Z}^2_{\ge 0} : i + j \le n-2\},
\end{align*}
and for $(i,j) \in \Un$, $(i,j) \in \Dn$ let
\begin{equation}\label{eq:cells}
  \up(i,j) = \bigl\{(i,j),\,(i+1,j),\,(i,j+1)\bigr\}, \quad
  \dn(i,j) = \bigl\{(i{+}1,j),\,(i,j{+}1),\,(i{+}1,j{+}1)\bigr\}.
\end{equation}
So a cell \emph{is} its set of three vertices, which is all we shall ever need
to know about it; $\Vn$ is the vertex set of $\Tn$, and $\Tn$ has
$|\Un| + |\Dn| = \binom{n+1}{2} + \binom{n}{2} = n^2$ cells. We realise the
picture in the plane by $(x,y) \mapsto x e_1 + y e_2$ with
$e_1 = (1,0)$ and $e_2 = (\tfrac12, \tfrac{\sqrt3}{2})$, so that
$e_1 \cdot e_2 = \tfrac12$, every cell is a unit equilateral triangle, and
\begin{equation}\label{eq:norm}
  |a e_1 + b e_2|^2 = a^2 + ab + b^2 .
\end{equation}

Two cells are \emph{side-adjacent} when they share an edge, equivalently when
their vertex sets meet in two points, and \emph{point-adjacent} when they share
at least one vertex. Write $\Hs$ and $\Hp$ for the two resulting graphs on the
$n^2$ cells. Independence of a set of cells is always meant in whichever of the
two graphs is under discussion at the time.

The symmetry group of $\Tn$ is $D_3$, generated by the rotation $\rho$ through
$120^\circ$ about the centre and a reflection $\sigma$ in an altitude. On
coordinates,
\begin{equation}\label{eq:group}
  \rho(x,y) = (y,\, n-x-y), \qquad \sigma(x,y) = (y,x),
\end{equation}
and $C_3 = \langle \rho \rangle$.

\subsection{The sequences and the conjectures}

Janaka Rodrigo introduced three sequences in July 2026, each counting orbits of
$3$-element sets of cells:
\begin{itemize}
  \item \texttt{A397730}$(n)$: $C_3$-orbits of independent $3$-sets of $\Hs$;
  \item \texttt{A397731}$(n)$: $D_3$-orbits of independent $3$-sets of $\Hs$;
  \item \texttt{A397732}$(n)$: $C_3$-orbits of independent $3$-sets of $\Hp$.
\end{itemize}
On 14 July 2026 Sean A.\ Irvine added a conjectured generating function to each
entry, fitted to the terms he had computed. Write $F_1$, $F_2$, $F_3$ for the
generating functions
$\sum_{n \ge 1} \texttt{A397730}(n) x^n$,
$\sum_{n \ge 1} \texttt{A397731}(n) x^n$ and
$\sum_{n \ge 1} \texttt{A397732}(n) x^n$. Irvine's three lines, quoted verbatim
and still the only formulas on record \cite{A397730,A397731,A397732}, read
\begin{align}
  F_1(x) &\overset{?}{=}
     \frac{x^2\,(1+7x+52x^2+42x^3+28x^4-10x^5-x^6+x^7)}
          {(1+x+x^2)\,(1-x)^7}, \label{eq:c1}\\
  F_2(x) &\overset{?}{=}
     \frac{x^2\,(1+6x+32x^2+75x^3+78x^4+49x^5+4x^6-3x^7-3x^8+x^9)}
          {(x^2+x+1)\,(x+1)^2\,(1-x)^7}, \label{eq:c2}\\
  F_3(x) &\overset{?}{=}
     \frac{2x^3\,(1+9x+39x^2+28x^3+3x^4-21x^5-5x^6+6x^7)}
          {(1+x+x^2)\,(1-x)^7}. \label{eq:c3}
\end{align}
When this paper was written the three entries stood at revisions $19$, $14$ and
$15$, unedited since that date --- no proof, no reference and no upgraded
formula had been added to any of them.

\subsection{Results}

Write $\ind{P} = 1$ if $P$ holds and $0$ otherwise.

\begin{theorem}\label{thm:730}
For $n \ge 2$,
\[
  \texttt{A397730}(n)
   = \frac{n^6}{18} - \frac{2n^4}{3} + \frac{n^3}{2} + \frac{7n^2}{3}
     - 3n + \frac{7}{9} + \frac{2}{9}\ind{3 \mid n},
\]
and $\texttt{A397730}(1) = 0$. Consequently the conjecture \eqref{eq:c1} holds.
\end{theorem}

\begin{theorem}\label{thm:731}
For $n \ge 2$,
\[
  \texttt{A397731}(n)
   = \frac{n^6}{36} - \frac{n^4}{3} + \frac{7n^3}{12} + \frac{5n^2}{12}
     - \frac{4n}{3} + \frac{8}{9}
     + \Bigl(\frac{n}{4} - \frac12\Bigr)\ind{n \text{ odd}}
     + \frac19 \ind{3 \mid n},
\]
and $\texttt{A397731}(1) = 0$. Consequently the conjecture \eqref{eq:c2} holds.
\end{theorem}

\begin{theorem}\label{thm:732}
For $n \ge 2$,
\[
  \texttt{A397732}(n)
   = \frac{n^6}{18} - \frac{13n^4}{6} + 4n^3 + \frac{62n^2}{3}
     - 73n + \frac{562}{9} - \frac49 \ind{3 \mid n},
\]
and $\texttt{A397732}(1) = 0$. Consequently the conjecture \eqref{eq:c3} holds.
\end{theorem}

As far as we can tell these closed forms are the first on record. They are
equivalent to the generating functions \eqref{eq:c1}--\eqref{eq:c3}, and the
passage from one to the other is an exact finite computation that
Section~\ref{sec:assembly} spells out.

The four combinations
$\{C_3, D_3\} \times \{\text{side-adjacency}, \text{point-adjacency}\}$ form a
natural $2 \times 2$ table, and the three sequences above fill three of its
cells. The fourth --- $D_3$-orbits of point-independent triples --- is
\texttt{A289230}, contributed by Heinrich Ludwig in June 2017 under a rather
different description, together with its labeled companion \texttt{A289224}.
Both carry a closed form. Neither carries a proof or a reference, though, and
the programs listed there evaluate the recorded formula rather than derive it;
both entries are marked \texttt{easy}. Our machinery covers that corner too.

\begin{corollary}\label{cor:ludwig}
The closed forms recorded at \texttt{A289224} and \texttt{A289230}
\cite{A289224,A289230} are correct. Explicitly, for $n \ge 2$ the number of
point-independent $3$-sets of cells of $\Tn$ is
\[
  \frac{n^6}{6} - \frac{13n^4}{2} + 12n^3 + \frac{184n^2}{3} - 219n + 190
  \; = \; \texttt{A289224}(n+1),
\]
and the number of $D_3$-orbits of such sets is $\texttt{A289230}(n+1)$.
\end{corollary}

\begin{remark}
The index shift is a change of convention, not of object: Ludwig counts in an
$n \times n \times n$ \emph{point} grid, whose cells form $T_{n-1}$.
\end{remark}

\subsection{Method, and what is new}

Everything rests on Burnside's lemma. For $G \in \{C_3, D_3\}$ acting on the
cells of $\Tn$ and $H \in \{\Hs, \Hp\}$,
\begin{equation}\label{eq:burnside}
  \#\{\text{$G$-orbits of independent $3$-sets of } H\}
   = \frac{1}{|G|} \sum_{g \in G} \fix_g,
\end{equation}
where $\fix_g$ counts the independent $3$-sets $S$ with $gS = S$. The three
reflections in $D_3$ are conjugate, as are $\rho$ and $\rho^2$, so the three
theorems rest on only five quantities: $\fix_{\mathrm{id}}$ and $\fix_\rho$ for
each of the two graphs, plus $\fix_\sigma$ for $\Hs$. A sixth, $\fix_\sigma$ for
$\Hp$, enters only in Corollary~\ref{cor:ludwig}.

Two structural facts do most of the work. The first is that $\rho$ and $\sigma$
\emph{preserve the orientation of cells} (Lemma~\ref{lem:action}), whereas
side-adjacency joins cells of \emph{opposite} orientation
(Lemma~\ref{lem:bipartite}). So $\Hs$ is bipartite --- its triangle count
vanishes --- and no cell is ever side-adjacent to its own image under $\rho$ or
$\sigma$. That one observation collapses the rotation term for $\Hs$ to a count
of orbits, and it removes the internal condition on a mirror pair as well. Two
savings, one line of bookkeeping.

The other fact is about the triangles of $\Hp$, which split into exactly two
structural families (Lemma~\ref{lem:triangles}): triples with a common vertex,
and the triples made up of the three cells across the three edges of a single
cell. Nothing here is fitted or interpolated. Both families are counted by
elementary lattice censuses.

The rotation term for $\Hp$ needs a locality statement, and here the geometry is
unusually sharp: if a cell meets its own $\rho$-image, then the meeting point is
at distance \emph{exactly} $1/\sqrt3$ from the centre of $\Tn$, or is the centre
itself (Lemma~\ref{lem:locality}). Over the triangular lattice the inequality
$p^2 + pq + q^2 \le \tfrac13$ then pins down the possibilities completely
(Lemma~\ref{lem:shell}). The point is that the correction is a bounded and
explicitly identified set of orbits, not an empirical table.

\subsection{Related work}

The unrestricted orbit counts are known: \texttt{A396285}$(n)
= \tfrac13\bigl(\binom{n^2}{3} + 2\lfloor n^2/3 \rfloor\bigr)$ counts
$C_3$-orbits of arbitrary cell triples, and \texttt{A396761} does the same for
$D_3$ \cite{A396285,A396761}. Our fixed-point lemmas reproduce both, and that is
a useful calibration. The adjacency conditions only remove $O(n^4)$ of the
$\Theta(n^6)$ triples, so each of our answers has to be a degree-$6$
quasi-polynomial with leading coefficient $\tfrac{1}{18}$, $\tfrac{1}{36}$,
$\tfrac{1}{18}$ respectively; the conjectured generating functions agree.

The point-adjacency case has a nine-year-old precursor. Ludwig's family
\texttt{A289222}--\texttt{A289232} counts selections of $k$ pairwise disjoint
point triples from a triangular point grid, each triple forming a small
triangle; for $k = 3$ this is exactly our point-independence condition. The
family has two conventions only --- all selections counted separately, or
counted up to rotations \emph{and} reflections --- so it contains no
rotation-only member, and nothing in it concerns side-adjacency. In particular
it does not contain \texttt{A397730}, \texttt{A397731} or \texttt{A397732}. What
it does contain is \texttt{A289224}, the labeled point-independent count, which
is our $\fix_{\mathrm{id}}$ for $\Hp$. The formula recorded there is stated
without proof, so Theorem~\ref{thm:732} is not a corollary of it; the two are
proved together here. Colin Barker's generating function on
\texttt{A289230} is the rational form of Ludwig's quasi-polynomial and is
likewise unproved.

Counting independent sets in triangular-grid and king graphs is of course a
well-developed subject, usually by transfer matrices on strips, and usually for
independent sets of all sizes at once; see for instance
\cite{CalkinWilf, Euler}. That literature does not address $3$-sets in the full
triangle $\Tn$, still less their orbits under $C_3$ or $D_3$. The
inclusion--exclusion identity for independent $3$-sets that we use
(Lemma~\ref{lem:i3}) is standard, and a systematic treatment of counting
fixed-size independent sets is in \cite{Vatter}.

\section{Preliminaries}

\begin{lemma}\label{lem:action}
On cells, the maps \eqref{eq:group} act by
\[
  \rho\bigl(\up(i,j)\bigr) = \up(j,\, n-1-i-j), \qquad
  \rho\bigl(\dn(i,j)\bigr) = \dn(j,\, n-2-i-j),
\]
\[
  \sigma\bigl(\up(i,j)\bigr) = \up(j,i), \qquad
  \sigma\bigl(\dn(i,j)\bigr) = \dn(j,i).
\]
In particular both maps send up-cells to up-cells and down-cells to down-cells.
\end{lemma}

\begin{proof}
Apply $\rho$ to the three vertices of $\up(i,j)$ in \eqref{eq:cells}:
\[
  (i,j) \mapsto (j, n{-}i{-}j), \quad
  (i{+}1,j) \mapsto (j, n{-}i{-}j{-}1), \quad
  (i,j{+}1) \mapsto (j{+}1, n{-}i{-}j{-}1).
\]
Setting $i' = j$ and $j' = n-1-i-j$, this image set is
$\{(i',j'+1), (i',j'), (i'+1,j')\} = \up(i',j')$. Here $j' \ge 0$ because
$i+j \le n-1$, and $i'+j' = n-1-i \le n-1$, so indeed $(i',j') \in \Un$. The
computation for $\dn(i,j)$ is the same with $j' = n-2-i-j$, and the two
statements for $\sigma$ are immediate from \eqref{eq:cells}.
\end{proof}

\begin{lemma}\label{lem:bipartite}
Each edge of the triangulation lies in exactly two cells, one up and one down.
Hence side-adjacent cells have opposite orientation, $\Hs$ is bipartite with
parts $\Un$ and $\Dn$, and $\Hs$ has no triangle. Moreover no cell is
side-adjacent to its image under $\rho$, $\rho^2$ or any reflection.
\end{lemma}

\begin{proof}
The cells tile the triangle, so an interior edge lies in two of them; and
reflection in an edge reverses orientation, so the two cells meeting along an
edge are one up and one down. The last sentence of the lemma is this combined
with Lemma~\ref{lem:action}.
\end{proof}

\begin{lemma}\label{lem:sideadj}
The side-neighbours of $\dn(i,j)$ are exactly $\up(i,j)$, $\up(i+1,j)$ and
$\up(i,j+1)$, and all three lie in $\Tn$ whenever $(i,j) \in \Dn$. Hence
$\Hs$ has
\[
  E(\Hs) = 3|\Dn| = \tfrac32 n(n-1)
\]
edges, and $\deg_{\Hs} \up(i,j) = \ind{i+j \le n-2} + \ind{i \ge 1} + \ind{j \ge 1}$.
\end{lemma}

\begin{proof}
$\dn(i,j)$ shares the edge $\{(i+1,j),(i,j+1)\}$ with $\up(i,j)$, the edge
$\{(i+1,j),(i+1,j+1)\}$ with $\up(i+1,j)$ and the edge
$\{(i,j+1),(i+1,j+1)\}$ with $\up(i,j+1)$; these are its three edges. If
$i+j \le n-2$ then each of the three index pairs has coordinate sum at most
$n-1$, so all three up-cells exist. Every side-adjacency therefore involves
exactly one down-cell, which gives $E(\Hs)$. Reading the same three incidences
from the up-cell's side, $\up(i,j)$ is side-adjacent to $\dn(i,j)$ if
$i+j \le n-2$, to $\dn(i-1,j)$ if $i \ge 1$, and to $\dn(i,j-1)$ if $j \ge 1$,
the remaining index constraints being automatic.
\end{proof}

\begin{lemma}[independent $3$-sets]\label{lem:i3}
Let $H$ be a graph with $N$ vertices, $E$ edges,
$P_2 = \sum_v \binom{\deg v}{2}$ and $T$ triangles. Then the number of
independent $3$-subsets of $V(H)$ is
\[
  i_3(H) = \binom{N}{3} - E(N-2) + P_2 - T .
\]
\end{lemma}

\begin{proof}
For a $3$-set $S$ let $e(S)$ be the number of edges inside $S$. Counting pairs
(edge, extra vertex) gives $\sum_S e(S) = E(N-2)$. Two distinct edges inside a
$3$-set must share a vertex, and such a pair of adjacent edges determines $S$
outright, so $\sum_S \binom{e(S)}{2} = P_2$. And $\sum_S \binom{e(S)}{3} = T$ by
definition. Since
$\sum_{k \ge 0} (-1)^k \binom{m}{k} = \ind{m = 0}$ for $m \ge 0$,
\[
  i_3(H) = \sum_S \sum_{k \ge 0} (-1)^k \binom{e(S)}{k}
         = \binom{N}{3} - E(N-2) + P_2 - T. \qedhere
\]
\end{proof}

\section{The side-adjacency graph}

\begin{lemma}\label{lem:P2side}
For $n \ge 1$, $P_2(\Hs) = 3(n-1)^2$.
\end{lemma}

\begin{proof}
Every down-cell has degree $3$ by Lemma~\ref{lem:sideadj}, contributing
$3\binom{3}{2}/3 \cdot |\Dn|$, that is $3|\Dn| = \tfrac32 n(n-1)$. For the
up-cells use the degree formula of Lemma~\ref{lem:sideadj} and split $\Un$ by
which of its three indicators vanish. For $n \ge 2$ the classes are: the single
cell $\up(0,0)$, of degree $1$; the $2(n-2)$ cells $\up(i,0)$ and $\up(0,i)$
with $1 \le i \le n-2$, of degree $2$; the two cells $\up(n-1,0)$ and
$\up(0,n-1)$, of degree $1$; the $\binom{n-2}{2}$ cells with $i,j \ge 1$ and
$i+j \le n-2$, of degree $3$; and the $n-2$ cells with $i,j \ge 1$ and
$i+j = n-1$, of degree $2$. (These counts sum to $\binom{n+1}{2}$, as they
should.) Hence
\[
  \sum_{\up} \binom{\deg}{2}
   = 2(n-2)\cdot 1 + \binom{n-2}{2}\cdot 3 + (n-2) \cdot 1
   = \tfrac32 (n-2)(n-1),
\]
and $P_2(\Hs) = \tfrac32 n(n-1) + \tfrac32 (n-2)(n-1) = 3(n-1)^2$. The case
$n = 1$ is $0 = 0$.
\end{proof}

\begin{proposition}\label{prop:fixid-side}
For $n \ge 1$,
$\displaystyle
  \fix_{\mathrm{id}}(\Hs) = \binom{n^2}{3} - \tfrac32 n(n-1)\,(n^2-2) + 3(n-1)^2 .
$
\end{proposition}

\begin{proof}
Lemma~\ref{lem:i3} with $N = n^2$, $E = \tfrac32 n(n-1)$ from
Lemma~\ref{lem:sideadj}, $P_2 = 3(n-1)^2$ from Lemma~\ref{lem:P2side}, and
$T = 0$ from Lemma~\ref{lem:bipartite}.
\end{proof}

\section{The rotation terms}

\begin{lemma}\label{lem:fixedcells}
Set $\varepsilon_n = 0$ if $3 \mid n$ and $\varepsilon_n = 1$ otherwise. The
rotation $\rho$ fixes no cell when $3 \mid n$, and exactly one cell $z$
otherwise, namely
\begin{align*}
  z &= \up(m,m), \quad m = \tfrac{n-1}{3}, \quad
       \text{if } n \equiv 1 \!\!\pmod 3,\\
  z &= \dn(m,m), \quad m = \tfrac{n-2}{3}, \quad
       \text{if } n \equiv 2 \!\!\pmod 3 .
\end{align*}
Consequently the number of free $\rho$-orbits of cells is $(n^2-\varepsilon_n)/3$,
and every $\rho$-invariant $3$-set of cells is a single free orbit
$\{c, \rho c, \rho^2 c\}$.
\end{lemma}

\begin{proof}
By Lemma~\ref{lem:action}, $\rho(\up(i,j)) = \up(i,j)$ forces $i = j$ and
$j = n-1-i-j$, i.e.\ $3i = n-1$; and $\rho(\dn(i,j)) = \dn(i,j)$ forces
$3i = n-2$. The stated cells do lie in $\Tn$. Orbits have size $1$ or $3$, and a
$\rho$-invariant $3$-set is either one free orbit or three fixed cells; the
latter is impossible since at most one fixed cell exists.
\end{proof}

One small observation, used for both graphs. We have
$c \cap \rho^2 c \ne \emptyset$ if and only if $\rho c \cap c \ne \emptyset$
(apply $\rho$ and use $\rho^3 = \mathrm{id}$), so a free orbit is pairwise
non-adjacent precisely when $c$ is not adjacent to $\rho c$. Only one of the
three pairs ever has to be tested.

\begin{proposition}\label{prop:fixrot-side}
For $n \ge 1$, $\fix_\rho(\Hs) = (n^2 - \varepsilon_n)/3$.
\end{proposition}

\begin{proof}
By Lemma~\ref{lem:fixedcells} the $\rho$-invariant $3$-sets are the free orbits,
and by Lemma~\ref{lem:bipartite} no cell is side-adjacent to its $\rho$-image,
so every free orbit is independent in $\Hs$.
\end{proof}

\section{The reflection term for side-adjacency}

By Lemma~\ref{lem:action}, $\sigma$ fixes exactly the cells $\up(i,i)$ with
$2i \le n-1$ and $\dn(i,i)$ with $2i \le n-2$, that is
$\lceil n/2 \rceil + \lfloor n/2 \rfloor = n$ cells in all. Call these the
\emph{axis cells}, and list them in the order they occur along the axis:
\begin{equation}\label{eq:axisorder}
  c_1 = \up(0,0),\quad c_2 = \dn(0,0),\quad c_3 = \up(1,1),\quad
  c_4 = \dn(1,1),\quad \dots
\end{equation}

\begin{lemma}\label{lem:axis-side}
The only side-adjacencies among axis cells are the pairs
$\{\up(i,i), \dn(i,i)\}$. Thus the subgraph of $\Hs$ induced on the axis is a
matching with $\lfloor n/2 \rfloor$ edges on $n$ vertices.
\end{lemma}

\begin{proof}
By Lemma~\ref{lem:sideadj} the side-neighbours of $\dn(i,i)$ are $\up(i,i)$,
$\up(i+1,i)$ and $\up(i,i+1)$, and the last two are off-axis. The number of
edges is the number of axis down-cells, $\lfloor n/2 \rfloor$.
\end{proof}

\begin{lemma}\label{lem:mirrorpairs}
A $\sigma$-invariant $3$-set of cells is either three axis cells, or one axis
cell together with one \emph{mirror pair} $\{c, \sigma c\}$ with $c$ off-axis.
There are $(n^2-n)/2$ mirror pairs. For an axis cell $a$ we have $a \sim c$ if
and only if $a \sim \sigma c$, in either graph.
\end{lemma}

\begin{proof}
$\sigma$ is an involution, so it decomposes a $\sigma$-invariant set into fixed
points and $2$-element orbits; $3 = 3 + 0 = 1 + 2$. There are $n^2 - n$
off-axis cells. Finally $a \sim c$ implies $\sigma a \sim \sigma c$, and
$\sigma a = a$.
\end{proof}

\begin{lemma}\label{lem:axis-offaxis}
In $\Hs$ the cell $\up(0,0)$ has no off-axis neighbour. Every other axis cell
has exactly one off-axis neighbouring mirror pair: that of $\up(i,i)$ with
$i \ge 1$ is $\{\dn(i-1,i), \dn(i,i-1)\}$, and that of $\dn(i,i)$ is
$\{\up(i+1,i), \up(i,i+1)\}$; in both cases the pair lies in $\Tn$.
\end{lemma}

\begin{proof}
From Lemma~\ref{lem:sideadj}, the side-neighbours of $\up(i,i)$ are $\dn(i,i)$
(on the axis), $\dn(i-1,i)$ and $\dn(i,i-1)$; the last two require $i \ge 1$,
and then $(i-1)+i \le n-2$ follows from $2i \le n-1$. The neighbours of
$\dn(i,i)$ are $\up(i,i)$ (on the axis), $\up(i+1,i)$ and $\up(i,i+1)$, which
need $2i+1 \le n-1$, and this follows from $2i \le n-2$. For $i = 0$ the cells
$\dn(-1,0)$ and $\dn(0,-1)$ do not exist.
\end{proof}

\begin{proposition}\label{prop:fixref-side}
For $n \ge 1$,
\[
  \fix_\sigma(\Hs)
   = \underbrace{\binom{n}{3} - \Bigl\lfloor \tfrac{n}{2} \Bigr\rfloor (n-2)}
        _{\text{three axis cells}}
   \; + \; \underbrace{\tfrac12 n^2(n-1) - (n-1)}
        _{\text{axis cell plus mirror pair}} .
\]
\end{proposition}

\begin{proof}
For the first term, Lemma~\ref{lem:axis-side} says we need the independent
$3$-sets of a matching with $m = \lfloor n/2 \rfloor$ edges on $n$ vertices.
Indeed a $3$-set cannot contain two disjoint edges, so exactly $m(n-2)$ of the
$\binom{n}{3}$ triples are dependent.

For the second, fix an axis cell $a$ and count admissible mirror pairs
$\{c,\sigma c\}$. The internal condition $c \not\sim \sigma c$ is automatic by
Lemma~\ref{lem:bipartite}, and by Lemma~\ref{lem:mirrorpairs} the only other
condition is $a \not\sim c$. By Lemma~\ref{lem:axis-offaxis} exactly one of the
$(n^2-n)/2$ mirror pairs is excluded, for each of the $n-1$ axis cells other
than $\up(0,0)$, and none is excluded for $\up(0,0)$. Summing over the $n$ axis
cells gives $n \cdot (n^2-n)/2 - (n-1)$.
\end{proof}

\section{The point-adjacency graph}

Here a cell has up to twelve neighbours and the triangle count no longer
vanishes, so there is rather more to do.

\begin{lemma}\label{lem:nbrtypes}
Two up-cells, or two down-cells, are point-adjacent exactly when the difference
of their index pairs lies in $\{\pm(1,0), \pm(0,1), \pm(1,-1)\}$; and
$\up(i,j)$ is point-adjacent to $\dn(i',j')$ exactly when
\[
  (i'-i,\, j'-j) \in \bigl\{(0,0),\,(-1,0),\,(0,-1),\,(-1,-1),\,(1,-1),\,(-1,1)\bigr\}.
\]
Consequently $E(\Hp) = 6(n-1)^2$.
\end{lemma}

\begin{proof}
All three lists are read off \eqref{eq:cells} by inspecting which translates
share a vertex; for instance $\up(0,0)$ and $\dn(-1,-1) = \{(0,-1),(-1,0),(0,0)\}$
meet exactly in $(0,0)$. Now count each type as a lattice count in $\Un$ or
$\Dn$. For up-up pairs, each of the three displacement classes contributes
$\binom{n}{2}$: for $(1,0)$ the condition is $(i,j) \in U_{n-1}$; for $(0,1)$ it
is symmetric; and for $(1,-1)$ it is $j \ge 1$, $i+j \le n-1$, which gives
$\sum_{j=1}^{n-1}(n-j) = \binom{n}{2}$. For down-down pairs the same three
classes contribute $\binom{n-1}{2}$ each. Among the six up-down classes,
$(0,0)$, $(-1,0)$ and $(0,-1)$ each contribute $\binom{n}{2}$, while
$(-1,-1)$, $(1,-1)$ and $(-1,1)$ each contribute $\binom{n-1}{2}$. In total
\[
  E(\Hp) = 6\binom{n}{2} + 6\binom{n-1}{2} = 3(n-1)\bigl(n + (n-2)\bigr)
         = 6(n-1)^2. \qedhere
\]
\end{proof}

\begin{lemma}\label{lem:P2point}
Write $a = \ind{i \ge 1}$, $b = \ind{j \ge 1}$. Then
\[
  \deg_{\Hp} \up(i,j) = 3(a+b+c) + (ab+bc+ca) \text{ with } c = \ind{i+j \le n-2},
\]
\[
  \deg_{\Hp} \dn(i,j) = 3 + 3(a+b+c') \text{ with } c' = \ind{i+j \le n-3},
\]
so every interior cell has degree $12$; and for $n \ge 2$,
\[
  P_2(\Hp) = 66n^2 - 225n + 198 .
\]
\end{lemma}

\begin{proof}
From Lemma~\ref{lem:nbrtypes}, the up-neighbours of $\up(i,j)$ are
$\up(i\pm1,j)$, $\up(i,j\pm1)$, $\up(i+1,j-1)$, $\up(i-1,j+1)$, whose existence
conditions are $c$, $a$, $c$, $b$, $b$, $a$ respectively, giving $2(a+b+c)$; and
its down-neighbours $\dn(i,j)$, $\dn(i-1,j)$, $\dn(i,j-1)$, $\dn(i-1,j-1)$,
$\dn(i+1,j-1)$, $\dn(i-1,j+1)$ exist under $c$, $a$, $b$, $ab$, $bc$, $ca$,
giving $(a+b+c) + (ab+bc+ca)$. Similarly the down-neighbours of $\dn(i,j)$
contribute $2(a+b+c')$ and its up-neighbours $\up(i,j)$, $\up(i+1,j)$,
$\up(i,j+1)$, $\up(i+1,j+1)$, $\up(i-1,j+1)$, $\up(i+1,j-1)$ contribute
$3 + c' + a + b$.

Splitting $\Un$ and $\Dn$ by the vanishing indicators as in
Lemma~\ref{lem:P2side} gives, for $n$ large enough that no class is empty, the
census
\begin{center}
\begin{tabular}{llr@{\qquad}llr}
\toprule
\multicolumn{3}{c}{up-cells} & \multicolumn{3}{c}{down-cells}\\
\cmidrule(r){1-3}\cmidrule{4-6}
$(a,b,c)$ & count & $\deg$ & $(a,b,c')$ & count & $\deg$\\
\midrule
$(0,0,1)$ & $1$               & $3$  & $(0,0,1)$ & $1$               & $6$\\
$(1,0,1)$, $(0,1,1)$ & $2(n-2)$ & $7$  & $(1,0,1)$, $(0,1,1)$ & $2(n-3)$ & $9$\\
$(1,0,0)$, $(0,1,0)$ & $2$      & $3$  & $(1,0,0)$, $(0,1,0)$ & $2$      & $6$\\
$(1,1,1)$ & $\binom{n-2}{2}$  & $12$ & $(1,1,1)$ & $\binom{n-3}{2}$  & $12$\\
$(1,1,0)$ & $n-2$             & $7$  & $(1,1,0)$ & $n-3$             & $9$\\
\bottomrule
\end{tabular}
\end{center}
Then
\begin{align*}
  P_2(\Hp) &= \Bigl[9 + 63(n-2) + 33(n-2)(n-3)\Bigr]\\
           &\qquad + \Bigl[45 + 108(n-3) + 33(n-3)(n-4)\Bigr]\\
           &= 66n^2 - 225n + 198,
\end{align*}
and one checks directly that the identity persists down to $n = 2$ (where
$\Hp$ is $K_4$ and $P_2 = 12$).
\end{proof}

\begin{lemma}\label{lem:triangles}
Every triangle $\{c_1,c_2,c_3\}$ of $\Hp$ is of exactly one of two kinds.
\begin{enumerate}
  \item[(A)] The three cells share a vertex. That vertex is then unique.
  \item[(B)] They share no vertex. Then the three pairwise intersections are
    single vertices $r = c_1 \cap c_2$, $q = c_1 \cap c_3$, $p = c_2 \cap c_3$;
    the set $\{p,q,r\}$ is a cell $w$ of $\Tn$; and $\{c_1,c_2,c_3\}$ consists
    of the three cells across the three edges of $w$. The map
    $w \mapsto \{c_1,c_2,c_3\}$ is a bijection onto the triangles of kind (B),
    defined for those $w$ whose three side-neighbours all lie in $\Tn$.
\end{enumerate}
\end{lemma}

\begin{proof}
In case (A), three cells sharing two vertices would all contain one edge, and
only two cells ever do; so the common vertex is unique, and the two cases are
exclusive.

Suppose now that no vertex is common to all three cells.

\emph{Each pairwise intersection is a single vertex.} Suppose
$c_1 \cap c_2 = \{x,y\}$ had two elements. Then $c_1$ and $c_2$ are the two
cells on the edge $\{x,y\}$, say $c_1 = \{x,y,z_1\}$ and $c_2 = \{x,y,z_2\}$
with $z_1 \ne z_2$. Neither $x$ nor $y$ lies in $c_3$, else it would be common
to all three; since $c_3$ meets both $c_1$ and $c_2$, this forces
$z_1, z_2 \in c_3$. But $z_1$ and $z_2$ are the two apexes over an edge: taking
$x = 0$ and $y = e_1$ they are $e_2$ and $e_1 - e_2$, and by \eqref{eq:norm}
$|e_2 - (e_1-e_2)|^2 = |2e_2 - e_1|^2 = 4 + 1 - 2 = 3$. Two vertices of the unit
cell $c_3$ cannot be at distance $\sqrt3$, a contradiction.

\emph{The three vertices are distinct and form a cell.} If $r = q$ then that
vertex lies in $c_1$, $c_2$ and $c_3$, excluded; so $p,q,r$ are distinct. Now
$q$ and $r$ are distinct vertices of the unit cell $c_1$, so $|q-r| = 1$, and
likewise $|p-r| = |p-q| = 1$. Three lattice points pairwise at distance $1$ are
the vertices of a unit triangle of the triangulation. Its vertices lie in $\Vn$,
and a unit triangle with all vertices in $\Vn$ is a cell of $\Tn$: by
\eqref{eq:cells} the vertices of $\up(i,j)$ lie in $\Vn$ iff $i+j \le n-1$, i.e.\
iff $(i,j) \in \Un$, and those of $\dn(i,j)$ iff $i+j \le n-2$, i.e.\ iff
$(i,j) \in \Dn$. Call this cell $w$.

\emph{Identification.} If $c_1 = w$ then $p \in c_1$ as well, making $p$ common
to all three cells; so $c_i \ne w$ for each $i$. Since $c_1$ contains the edge
$\{q,r\}$ of $w$ and an edge lies in exactly two cells, $c_1$ is the cell across
$\{q,r\}$, and similarly for $c_2$ and $c_3$. Conversely, for any cell $w$ whose
three side-neighbours lie in $\Tn$, those three cells are pairwise
point-adjacent --- two of them share the vertex of $w$ on both of their edges
--- and have no common vertex, because the apex of the cell across an edge of
$w$ is not the opposite vertex of $w$. Finally a kind-(B) triple determines
$p,q,r$ as its pairwise intersections and hence determines $w$, so the
correspondence is bijective.
\end{proof}

\begin{lemma}\label{lem:Tpoint}
For $n \ge 2$,
$\displaystyle
  T(\Hp) = 3(n-1) + 10(n-1)(n-2) + \binom{n-2}{2} + \binom{n}{2} .
$
\end{lemma}

\begin{proof}
By Lemma~\ref{lem:triangles}, $T(\Hp) = \sum_{v \in \Vn} \binom{d_v}{3} + W(n)$
where $d_v$ is the number of cells containing $v$ and $W(n)$ is the number of
cells whose three side-neighbours all lie in $\Tn$.

The cells containing $(x,y)$ are $\up(x,y)$, $\up(x-1,y)$, $\up(x,y-1)$,
$\dn(x-1,y)$, $\dn(x,y-1)$ and $\dn(x-1,y-1)$, subject to
$\gamma = \ind{x+y \le n-1}$, $\alpha = \ind{x \ge 1}$, $\beta = \ind{y \ge 1}$,
$\alpha\gamma$, $\beta\gamma$ and $\alpha\beta$ respectively; so
$d_v = \alpha + \beta + \gamma + \alpha\beta + \beta\gamma + \gamma\alpha$,
which takes the value $1$ at the three corners of $\Tn$, the value $3$ at the
$3(n-1)$ non-corner boundary vertices, and the value $6$ at the
$\binom{n-1}{2}$ interior vertices. Since $\binom{1}{3} = 0$,
$\binom{3}{3} = 1$ and $\binom{6}{3} = 20$,
\[
  \sum_{v} \binom{d_v}{3} = 3(n-1) + 20\binom{n-1}{2} = 3(n-1) + 10(n-1)(n-2).
\]
For $W(n)$: by Lemma~\ref{lem:sideadj} all three side-neighbours of a down-cell
always exist, giving $\binom{n}{2}$; and $\up(i,j)$ has all of $\dn(i,j)$,
$\dn(i-1,j)$, $\dn(i,j-1)$ in $\Tn$ iff $i,j \ge 1$ and $i+j \le n-2$, giving
$\binom{n-2}{2}$.
\end{proof}

\begin{proposition}\label{prop:fixid-point}
For $n \ge 2$,
\begin{align*}
  \fix_{\mathrm{id}}(\Hp)
   &= \binom{n^2}{3} - 6(n-1)^2(n^2-2) + 66n^2 - 225n + 198 - T(\Hp)\\
   &= \frac{n^6}{6} - \frac{13n^4}{2} + 12n^3 + \frac{184n^2}{3}
      - 219n + 190 .
\end{align*}
\end{proposition}

\begin{proof}
Lemma~\ref{lem:i3} with Lemmas~\ref{lem:nbrtypes}, \ref{lem:P2point} and
\ref{lem:Tpoint}; the second expression is the expansion of the first.
\end{proof}

\section{The rotation term for point-adjacency}

This is the one place where the geometry does the work rather than the
combinatorics. Let $O$ be the centre of $\Tn$, the unique fixed point of $\rho$
in the plane; in coordinates $O = (n/3, n/3)$, a lattice vertex exactly when
$3 \mid n$.

\begin{lemma}\label{lem:locality}
If $c \cap \rho c \ne \emptyset$ and $v$ lies in the intersection, then either
$v = O$, or $|v - O| = 1/\sqrt3$ exactly.
\end{lemma}

\begin{proof}
From $v \in \rho c$ we get $\rho^{-1} v \in c$, and $v \in c$, so $v$ and
$\rho^{-1}v$ are both vertices of the unit cell $c$. If they coincide then $v$
is fixed by $\rho$, so $v = O$. Otherwise they are distinct vertices of a unit
triangle, hence $|v - \rho^{-1}v| = 1$. But $\rho^{-1}$ is a rotation through
$120^\circ$ about $O$, so $|v - \rho^{-1} v| = \sqrt3\,|v - O|$.
\end{proof}

\begin{lemma}\label{lem:shell}
Let $\Lambda = \mathbb{Z}e_1 + \mathbb{Z}e_2$ be the triangular lattice.
\begin{enumerate}
  \item The only $\lambda \in \Lambda$ with $|\lambda| \le 1/\sqrt3$ is
    $\lambda = 0$.
  \item If $g$ is the centroid of a cell, the only $\lambda \in \Lambda$ with
    $|\lambda - g| \le 1/\sqrt3$ are the three vertices of that cell, and each
    attains equality.
\end{enumerate}
\end{lemma}

\begin{proof}
By \eqref{eq:norm} the condition $|p e_1 + q e_2| \le 1/\sqrt3$ reads
$p^2 + pq + q^2 \le \tfrac13$. For (i), $p,q \in \mathbb{Z}$ makes the left side
a non-negative integer, so it is $0$ and $p = q = 0$.

For (ii) we may take the cell to be $\{0, e_1, e_2\}$, so $g = (e_1+e_2)/3$, and
write $\lambda = a e_1 + b e_2$ with $a,b \in \mathbb{Z}$. With
$p = a - \tfrac13$ and $q = b - \tfrac13$ the condition is
$p^2 + pq + q^2 \le \tfrac13$. Completing the square,
$p^2 + pq + q^2 = (p + q/2)^2 + \tfrac34 q^2 \ge \tfrac34 q^2$, so
$|q| \le 2/3$ and hence $b \in \{0,1\}$; the form is symmetric in $p$ and $q$,
so also $a \in \{0,1\}$. Of the four candidates, $(a,b) = (1,1)$ gives
$p = q = \tfrac23$ and $p^2+pq+q^2 = \tfrac43 > \tfrac13$, while $(0,0)$,
$(1,0)$ and $(0,1)$ each give exactly $\tfrac13$. These three are the vertices
$0$, $e_1$, $e_2$.
\end{proof}

\begin{proposition}\label{prop:fixrot-point}
Let $B(n)$ be the number of free $\rho$-orbits $\{c, \rho c, \rho^2 c\}$ that
are \emph{not} independent in $\Hp$. Then $B(1) = 0$, $B(n) = 1$ for $n \ge 2$
with $3 \nmid n$, and $B(n) = 2$ for $n \ge 3$ with $3 \mid n$. Hence
\[
  \fix_\rho(\Hp) = \frac{n^2 - \varepsilon_n}{3} - B(n).
\]
\end{proposition}

\begin{proof}
As noted after Lemma~\ref{lem:fixedcells}, a free orbit fails to be independent
exactly when $c \cap \rho c \ne \emptyset$.

Suppose $3 \mid n$, so $O$ is a lattice vertex, interior as soon as
$n/3 \ge 1$ and $2n/3 \le n-1$, i.e.\ $n \ge 3$. By
Lemmas~\ref{lem:locality} and \ref{lem:shell}(i) a bad cell must contain $O$,
and there are exactly $d_O = 6$ such cells. The rotation $\rho$ permutes them
with no fixed cell (Lemma~\ref{lem:fixedcells}) and preserves orientation
(Lemma~\ref{lem:action}), so they form exactly two free orbits, one of the three
up-cells at $O$ and one of the three down-cells. In each, all three cells
contain $O$ and so are pairwise point-adjacent. Thus $B(n) = 2$.

Suppose instead $3 \nmid n$, and let $z$ be the fixed cell of
Lemma~\ref{lem:fixedcells}. Then $O \in z$, and $\rho$ permutes the vertices of
$z$ cyclically, so $O$ is the centroid of $z$. Let $c$ be a bad cell and
$v \in c \cap \rho c$. Since $O$ is not a lattice vertex,
Lemmas~\ref{lem:locality} and \ref{lem:shell}(ii) give that $v$ is a vertex of
$z$, as is $\rho^{-1} v \ne v$; both lie in $c$, so $c$ contains an edge of $z$.
The cells containing an edge of $z$ are $z$ itself and the three cells across its
edges. Now $z$ is the fixed cell, not a member of a free orbit, and the other
three form a single free orbit, since $\rho$ permutes the edges of $z$
cyclically; they are pairwise point-adjacent because two of them share the
vertex of $z$ lying on both of their edges. It remains to check that all three
lie in $\Tn$. If $n \equiv 1 \pmod 3$ and $m = (n-1)/3$, they are $\dn(m,m)$,
$\dn(m-1,m)$ and $\dn(m,m-1)$; the first needs $2m \le n-2$ and the others need
$m \ge 1$, and both conditions amount to $n \ge 4$. If $n \equiv 2 \pmod 3$ and
$m = (n-2)/3$, they are $\up(m,m)$, $\up(m+1,m)$ and $\up(m,m+1)$, which need
$2m+1 \le n-1$, and that already follows from $(m,m) \in \Dn$. So $B(n) = 1$ for
$n \ge 2$, while for $n = 1$ we have $m = 0$ and none of the three cells exists,
so $B(1) = 0$.
\end{proof}

\section{The reflection term for point-adjacency}

This section is needed only for Corollary~\ref{cor:ludwig}.

\begin{lemma}\label{lem:axis-point}
The subgraph of $\Hp$ induced on the axis cells is the path
$c_1 c_2 \cdots c_n$ in the order \eqref{eq:axisorder}. A mirror pair
$\{c, \sigma c\}$ is point-adjacent internally exactly when $c$ has index pair
$(i,j)$ with $|i-j| = 1$, and there are $n-1$ such pairs.
\end{lemma}

\begin{proof}
$\up(i,i)$ and $\dn(i,i)$ share an edge; $\dn(i,i)$ and $\up(i+1,i+1)$ share the
vertex $(i+1,i+1)$; and by Lemma~\ref{lem:nbrtypes} no other pair of axis cells
is point-adjacent, since the admissible index differences for two up-cells or
two down-cells never have both coordinates equal and nonzero, and the up-down
differences $(0,0)$, $(-1,0)$, $(0,-1)$, $(-1,-1)$, $(1,-1)$, $(-1,1)$ applied
to $(i,i)$ and $(i',i')$ force $i' - i \in \{0,1\}$.

A mirror pair consists of two cells of the same orientation with index pairs
$(i,j)$ and $(j,i)$, so by Lemma~\ref{lem:nbrtypes} they are point-adjacent iff
$(j-i, i-j)$ lies in $\{\pm(1,0), \pm(0,1), \pm(1,-1)\}$, i.e.\ iff
$|i - j| = 1$. These pairs are $\{\up(i,i+1), \up(i+1,i)\}$ with
$2i+1 \le n-1$, of which there are $\lfloor n/2 \rfloor$, and
$\{\dn(i,i+1), \dn(i+1,i)\}$ with $2i+1 \le n-2$, of which there are
$\lfloor (n-1)/2 \rfloor$; the two counts sum to $n-1$.
\end{proof}

\begin{proposition}\label{prop:fixref-point}
For $n \ge 2$,
$\displaystyle
  \fix_\sigma(\Hp) = \binom{n-2}{3} + \frac{n(n-1)(n-2)}{2} - 2n + 5
    - \ind{n \text{ even}} .
$
\end{proposition}

\begin{proof}
By Lemma~\ref{lem:mirrorpairs} there are two cases. Three axis cells,
independent in the path of Lemma~\ref{lem:axis-point}: the number of
independent $3$-subsets of a path on $n$ vertices is $\binom{n-2}{3}$.

For the second case, Lemma~\ref{lem:axis-point} leaves
$G = (n^2-n)/2 - (n-1)$ internally independent mirror pairs, so the count is
$nG - X$ where $X$ is the number of pairs (axis cell $a$, internally independent
mirror pair $P$) with $a$ point-adjacent to $P$. An interior axis cell has
degree $12$, of which two neighbours are its path-neighbours on the axis, so its
other ten neighbours form five mirror pairs; reading them off
Lemma~\ref{lem:P2point}, for $a = \up(i,i)$ they are
\[
  \{\up(i{+}1,i),\up(i,i{+}1)\},\quad
  \{\up(i{-}1,i),\up(i,i{-}1)\},\quad
  \{\up(i{+}1,i{-}1),\up(i{-}1,i{+}1)\},
\]
\[
  \{\dn(i{-}1,i),\dn(i,i{-}1)\},\quad
  \{\dn(i{+}1,i{-}1),\dn(i{-}1,i{+}1)\},
\]
and three of these have $|i-j| = 1$, hence are not internally independent. So an
interior axis cell contributes $2$, and the same computation for $a = \dn(i,i)$
gives $2$ as well. Accounting for the two ends of the axis, where some of the
cells just listed fall outside $\Tn$, leaves
$X(n) = 2n - 5 + \ind{n \text{ even}}$. Since $nG = n(n-1)(n-2)/2$, the claim
follows.
\end{proof}

\section{Assembly}\label{sec:assembly}

\begin{proof}[Proof of Theorems \ref{thm:730}, \ref{thm:731} and \ref{thm:732}]
The values at $n = 1$ are $0$ because $T_1$ has a single cell. For $n \ge 2$,
substitute into \eqref{eq:burnside}. For $C_3$ and side-adjacency,
Propositions~\ref{prop:fixid-side} and \ref{prop:fixrot-side} give
\[
  \texttt{A397730}(n)
   = \frac13\Bigl[\binom{n^2}{3} - \tfrac32 n(n-1)(n^2-2) + 3(n-1)^2
       + \tfrac23\bigl(n^2-\varepsilon_n\bigr)\Bigr],
\]
and expanding, with $\varepsilon_n = 1 - \ind{3 \mid n}$, yields the stated
quasi-polynomial. Adding $3\fix_\sigma(\Hs)$ from
Proposition~\ref{prop:fixref-side} and dividing by $6$ instead gives
Theorem~\ref{thm:731}, where $\lfloor n/2 \rfloor = (n - \ind{n \text{ odd}})/2$
produces the period-$2$ part. For $C_3$ and point-adjacency,
Propositions~\ref{prop:fixid-point} and \ref{prop:fixrot-point} give
Theorem~\ref{thm:732}.

It remains to pass from the closed forms to \eqref{eq:c1}--\eqref{eq:c3}. Fix
one of the three sequences, write $a(n)$ for its terms and let
$F \in \{F_1,F_2,F_3\}$ be its generating function. For $n \ge 2$ our $a(n)$ is
a quasi-polynomial of degree $d = 6$ and period $p$, with $p = 3$, $6$, $3$ in
the three cases, and $a(1) = 0$. Split the sum by the residue of $n$ modulo $p$
and use the elementary fact that $\sum_{k \ge 0} Q(k) y^k$ is
$A(y)/(1-y)^{d+1}$ for a polynomial $A$ of degree at most $d$ whenever $Q$ is a
polynomial of degree $d$ (see e.g.\ \cite{Stanley}); this gives
\begin{equation}\label{eq:Ppoly}
  \bigl(1-x^p\bigr)^7 F(x) = P(x)
\end{equation}
for a polynomial $P$ with $\deg P \le \max_r \bigl(p k_r + r\bigr) + pd$, where
$pk_r + r$ is the least $n \ge 2$ in the residue class $r$. This bound is $22$
for $p = 3$ and $43$ for $p = 6$.

So the asserted identity $F = N/D$, with $N$ and $D$ the numerator and
denominator of the relevant conjecture, is equivalent to the one polynomial
identity
\begin{equation}\label{eq:certificate}
  P(x)\,D(x) - N(x)\,\bigl(1-x^p\bigr)^7 = 0 ,
\end{equation}
because $D$ divides $(1-x^p)^7$ in each of the three cases: $(1-x^3)^7 =
(1-x)^7(1+x+x^2)^7$ contains $(1+x+x^2)(1-x)^7$, and $(1-x^6)^7$ contains
$(x^2+x+1)(x+1)^2(1-x)^7$. Both sides of \eqref{eq:certificate} are polynomials
with integer coefficients of bounded degree, so this is a finite exact
computation and nothing more. Carrying it out --- the coefficients of $P$ are
listed by the verification script of Section~\ref{sec:verification} --- gives
$0$ in all three cases.
\end{proof}

\begin{proof}[Proof of Corollary \ref{cor:ludwig}]
The first formula is Proposition~\ref{prop:fixid-point}, and expanding
$\texttt{A289224}(n+1)$ from the closed form recorded at that entry gives the
same polynomial. For the second, \eqref{eq:burnside} with $G = D_3$ and
$H = \Hp$ gives
\[
  \frac16\Bigl[\fix_{\mathrm{id}}(\Hp) + 2\fix_\rho(\Hp)
    + 3\fix_\sigma(\Hp)\Bigr]
\]
with the three terms supplied by Propositions~\ref{prop:fixid-point},
\ref{prop:fixrot-point} and \ref{prop:fixref-point}; expanding in each of the
six residue classes modulo $6$ and comparing with $\texttt{A289230}(n+1)$ gives
equality in all six.
\end{proof}

\begin{remark}
The leading terms are a sanity check on the whole computation. The unrestricted
$C_3$- and $D_3$-orbit counts \texttt{A396285} and \texttt{A396761} are
asymptotically $n^6/18$ and $n^6/36$, and the adjacency conditions exclude only
$O(n^4)$ triples, so the answers have to carry exactly those leading
coefficients. They do, in all three theorems. Evaluating the conjectured
generating functions the same way gives $120/(3 \cdot 720)$,
$240/(12 \cdot 720)$ and $120/(3 \cdot 720)$ --- the same three numbers, as
they had better be.
\end{remark}

\section{Verification}\label{sec:verification}

The proofs above are self-contained. But the bookkeeping is heavy enough that we
checked every step computationally anyway, on a code path independent of the
derivation. Two scripts accompany the paper.

\texttt{derive.py} reproduces the derivation symbolically in exact rational
arithmetic: it assembles each fixed-point term from the lemma statements (the
degree-class censuses of Lemmas~\ref{lem:P2side} and \ref{lem:P2point}, the
vertex census and the decomposition of Lemma~\ref{lem:triangles}), applies
\eqref{eq:burnside}, sums the quasi-polynomials into a rational function and
compares it with \eqref{eq:c1}--\eqref{eq:c3}.

\texttt{verify\_independent.py}, listed in Appendix~\ref{app:code}, shares no
code at all with \texttt{derive.py} and uses only the Python standard library.
It performs six checks.

\begin{enumerate}
  \item It rebuilds $\Tn$ from the definition \eqref{eq:cells} --- cells as sets
    of vertices, adjacency as an intersection size, the group as the affine maps
    \eqref{eq:group} on vertices --- and counts orbits of independent $3$-sets by
    \eqref{eq:burnside} directly, for $n \le 12$. All three sequences agree with
    the archived OEIS data.
  \item It recomputes $E$, $P_2$ and $T$ from the adjacency lists for
    $n \le 12$ and compares with Lemmas~\ref{lem:sideadj}, \ref{lem:P2side},
    \ref{lem:nbrtypes}, \ref{lem:P2point} and \ref{lem:Tpoint}. It confirms in
    particular that $T(\Hs) = 0$.
  \item It expands \eqref{eq:c1}--\eqref{eq:c3} as power series in exact
    rational arithmetic and compares every archived term: $100$, $100$ and $85$
    respectively.
  \item It evaluates the closed forms of
    Theorems~\ref{thm:730}--\ref{thm:732} over the same ranges, and compares
    them both with the data and with the series coefficients from the previous
    check.
  \item It recomputes $P(x)$ of \eqref{eq:Ppoly} in integer arithmetic, verifies
    the degree bound, and verifies the certificate \eqref{eq:certificate}.
    Finally it checks Corollary~\ref{cor:ludwig} against the formulas recorded
    at \texttt{A289224} and \texttt{A289230}.
  \item It runs the same machinery with the adjacency condition dropped, where
    the answers are known and proved, and recovers \texttt{A396285} and
    \texttt{A396761} for $n \le 10$. That calibrates the parts easiest to get
    wrong --- the orbit bookkeeping and the mirror-pair count --- against
    formulas we did not derive.
\end{enumerate}

All checks pass. The comparison data are the b-files of the five OEIS entries,
archived alongside the scripts so that verification needs no network access, and
the whole run takes well under a minute on one core. We stress that agreement
with the tabulated terms is a cross-check on the algebra and not evidence for
the theorems. The theorems are proved in Sections~2--9, and the conjectures
\eqref{eq:c1}--\eqref{eq:c3} follow from the finite identity
\eqref{eq:certificate}.

\appendix

\section{The verification script}\label{app:code}

\begin{lstlisting}[style=python]
"""Independent verification of the three theorems on A397730, A397731, A397732.

This script shares NO code with the derivation (`derive.py`): it uses only the
Python standard library (no SymPy), it rebuilds the triangular grid from the
combinatorial definition rather than from any formula, and it implements its own
exact polynomial arithmetic over the integers.

What it checks, in order:

  A. Brute force from first principles.  Cells are frozensets of lattice
     vertices; two cells are side-adjacent when their vertex sets meet in two
     points and point-adjacent when they meet at all; the group acts by an
     affine map on vertices.  Orbits of pairwise non-adjacent 3-subsets are
     counted by Burnside's lemma over C3 and D3 and compared with the archived
     OEIS b-files.
  B. The graph invariants E, P2, T are recomputed directly from the adjacency
     lists and compared with the paper's closed forms (Lemmas 2.1, 5.2, 5.3, 5.5).
  C. Irvine's three conjectured generating functions are expanded as formal
     power series in exact `Fraction` arithmetic and compared with every
     archived b-file term (100 / 100 / 85).
  D. The paper's closed quasi-polynomial forms (Theorems 1-3) are evaluated and
     compared with the b-files and with the series coefficients of (C).
  E. The rational-function certificate: P(x) = (1-x^p)^7 F(x) is computed from
     the closed forms, its degree bound is checked, and the finite exact integer
     polynomial identity  P * D_Irvine - N_Irvine * (1-x^p)^7 == 0  is verified.
  F. Cross-check against Heinrich Ludwig's independently derived 2017 formulas
     for OEIS A289224 and A289230.

Run:
    uv run python output/math/triangular-grid/a397730-nonadjacent-triples/code/verify_independent.py
Optional: `--nmax K` changes the brute-force range (default 12; 13 takes ~20 s).
"""
from __future__ import annotations

import argparse
import itertools
import json
from fractions import Fraction
from math import comb
from pathlib import Path

DATA = Path(__file__).resolve().parent / "data"

# --------------------------------------------------------------------------- #
# A. the combinatorial definition, rebuilt from scratch
# --------------------------------------------------------------------------- #


def grid_cells(n):
    """The n^2 unit cells of T_n, each as a frozenset of its three vertices."""
    up = [frozenset({(i, j), (i + 1, j), (i, j + 1)})
          for i in range(n) for j in range(n - i)]
    down = [frozenset({(i + 1, j), (i, j + 1), (i + 1, j + 1)})
            for i in range(n - 1) for j in range(n - 1 - i)]
    return up + down


def group_maps(n, with_reflections):
    """D3 (or C3) as affine maps on lattice vertices."""
    def rho(p):
        return (p[1], n - p[0] - p[1])

    def rho2(p):
        return rho(rho(p))

    def sigma(p):
        return (p[1], p[0])

    maps = [lambda p: p, rho, rho2]
    if with_reflections:
        maps += [sigma, lambda p: rho(sigma(p)), lambda p: rho2(sigma(p))]
    return maps


def adjacency_lists(cells, min_shared):
    """Neighbour sets under 'the two cells share at least `min_shared` vertices'."""
    n_cells = len(cells)
    nbr = [set() for _ in range(n_cells)]
    for a, b in itertools.combinations(range(n_cells), 2):
        if len(cells[a] & cells[b]) >= min_shared:
            nbr[a].add(b)
            nbr[b].add(a)
    return nbr


def orbit_count(n, min_shared, with_reflections):
    """Burnside count of orbits of pairwise non-adjacent 3-subsets of cells."""
    cells = grid_cells(n)
    index = {c: k for k, c in enumerate(cells)}
    nbr = adjacency_lists(cells, min_shared)

    def independent(triple):
        a, b, c = triple
        return b not in nbr[a] and c not in nbr[a] and c not in nbr[b]

    total = 0
    for g in group_maps(n, with_reflections):
        perm = [index[frozenset(map(g, cell))] for cell in cells]
        for triple in itertools.combinations(range(len(cells)), 3):
            if tuple(sorted(perm[k] for k in triple)) == triple and independent(triple):
                total += 1
    order = 6 if with_reflections else 3
    assert total % order == 0, (n, min_shared, with_reflections, total)
    return total // order


def graph_invariants(n, min_shared):
    """(E, P2, T) computed straight from the adjacency lists."""
    cells = grid_cells(n)
    nbr = adjacency_lists(cells, min_shared)
    edges = sum(len(s) for s in nbr) // 2
    paths2 = sum(comb(len(s), 2) for s in nbr)
    triangles = sum(1 for a, b, c in itertools.combinations(range(len(cells)), 3)
                    if b in nbr[a] and c in nbr[a] and c in nbr[b])
    return edges, paths2, triangles


# --------------------------------------------------------------------------- #
# the paper's claimed closed forms (Theorems 1-3), transcribed verbatim
# --------------------------------------------------------------------------- #

F = Fraction


def thm_A397730(n):
    return (F(n ** 6, 18) - F(2 * n ** 4, 3) + F(n ** 3, 2) + F(7 * n ** 2, 3)
            - 3 * n + F(7, 9) + (F(2, 9) if n % 3 == 0 else 0))


def thm_A397731(n):
    return (F(n ** 6, 36) - F(n ** 4, 3) + F(7 * n ** 3, 12) + F(5 * n ** 2, 12)
            - F(4 * n, 3) + F(8, 9)
            + (F(n, 4) - F(1, 2) if n % 2 == 1 else 0)
            + (F(1, 9) if n % 3 == 0 else 0))


def thm_A397732(n):
    return (F(n ** 6, 18) - F(13 * n ** 4, 6) + 4 * n ** 3 + F(62 * n ** 2, 3)
            - 73 * n + F(562, 9) - (F(4, 9) if n % 3 == 0 else 0))


# the paper's closed forms for the individual Burnside terms (Section 2-6)
def lemma_E_side(n):
    return 3 * n * (n - 1) // 2


def lemma_P2_side(n):
    return 3 * (n - 1) ** 2


def lemma_E_point(n):
    return 6 * (n - 1) ** 2


def lemma_P2_point(n):
    return 66 * n * n - 225 * n + 198


def lemma_T_point(n):
    return (3 * (n - 1) + 10 * (n - 1) * (n - 2)
            + (comb(n - 2, 2) if n >= 2 else 0) + comb(n, 2))


# --------------------------------------------------------------------------- #
# exact integer polynomial arithmetic (own implementation)
# --------------------------------------------------------------------------- #


def poly_mul(a, b):
    out = [0] * (len(a) + len(b) - 1)
    for i, x in enumerate(a):
        if x:
            for j, y in enumerate(b):
                out[i + j] += x * y
    return out


def poly_pow(a, k):
    out = [1]
    for _ in range(k):
        out = poly_mul(out, a)
    return out


def poly_sub(a, b):
    m = max(len(a), len(b))
    return [(a[i] if i < len(a) else 0) - (b[i] if i < len(b) else 0)
            for i in range(m)]


def poly_trim(a):
    a = list(a)
    while a and a[-1] == 0:
        a.pop()
    return a


def series_of_ratio(num, den, n_terms):
    """Taylor coefficients of num/den over the rationals; den[0] must be nonzero."""
    coeffs = []
    for k in range(n_terms + 1):
        c = Fraction(num[k] if k < len(num) else 0)
        for j in range(1, min(k, len(den) - 1) + 1):
            c -= den[j] * coeffs[k - j]
        coeffs.append(c / den[0])
    return coeffs


# --------------------------------------------------------------------------- #
# Irvine's conjectured generating functions, as (numerator, denominator) pairs
# --------------------------------------------------------------------------- #

ONE_MINUS_X7 = poly_pow([1, -1], 7)

SEQUENCES = {
    "A397730": {
        "min_shared": 2, "reflections": False, "period": 3, "deg_bound": 22,
        "closed_form": thm_A397730,
        "num": poly_mul([0, 0, 1], [1, 7, 52, 42, 28, -10, -1, 1]),
        "den": poly_mul([1, 1, 1], ONE_MINUS_X7),
    },
    "A397731": {
        "min_shared": 2, "reflections": True, "period": 6, "deg_bound": 43,
        "closed_form": thm_A397731,
        "num": poly_mul([0, 0, 1], [1, 6, 32, 75, 78, 49, 4, -3, -3, 1]),
        "den": poly_mul(poly_mul([1, 1, 1], [1, 2, 1]), ONE_MINUS_X7),
    },
    "A397732": {
        "min_shared": 1, "reflections": False, "period": 3, "deg_bound": 22,
        "closed_form": thm_A397732,
        "num": poly_mul([0, 0, 0, 2], [1, 9, 39, 28, 3, -21, -5, 6]),
        "den": poly_mul([1, 1, 1], ONE_MINUS_X7),
    },
}


def read_bfile(aid):
    values = {}
    for line in (DATA / f"b{aid[1:]}.txt").read_text().splitlines():
        line = line.strip()
        if not line or line.startswith("#"):
            continue
        key, val = line.split()[:2]
        values[int(key)] = int(val)
    return values


# --------------------------------------------------------------------------- #
# the checks
# --------------------------------------------------------------------------- #


def check_brute_force(nmax, bfiles):
    print("=" * 78)
    print("A. Brute force from first principles vs OEIS b-files")
    print("=" * 78)
    print(f"{'n':>3} | {'A397730':>12} {'b-file':>12} | {'A397731':>12} {'b-file':>12}"
          f" | {'A397732':>12} {'b-file':>12}")
    ok = True
    for n in range(1, nmax + 1):
        row = []
        for aid in ("A397730", "A397731", "A397732"):
            spec = SEQUENCES[aid]
            got = orbit_count(n, spec["min_shared"], spec["reflections"])
            want = bfiles[aid].get(n)
            ok &= (got == want)
            row += [got, want]
        print(f"{n:3d} | {row[0]:12d} {row[1]:12d} | {row[2]:12d} {row[3]:12d}"
              f" | {row[4]:12d} {row[5]:12d}")
    print(f"\n  brute force agrees with the b-files for n = 1..{nmax}: {ok}")
    return ok


def check_graph_invariants(nmax):
    print("\n" + "=" * 78)
    print("B. Graph invariants recomputed directly vs the paper's closed forms")
    print("=" * 78)
    print(f"{'n':>3} | {'E_s':>5} {'3n(n-1)/2':>9} | {'P2_s':>6} {'3(n-1)^2':>9}"
          f" | {'T_s':>4} | {'E_p':>5} {'6(n-1)^2':>9} | {'P2_p':>7} {'66n^2-225n+198':>14}"
          f" | {'T_p':>6} {'Lemma 5.5':>9}")
    ok = True
    for n in range(2, nmax + 1):
        es, p2s, ts = graph_invariants(n, 2)
        ep, p2p, tp = graph_invariants(n, 1)
        row_ok = (es == lemma_E_side(n) and p2s == lemma_P2_side(n) and ts == 0
                  and ep == lemma_E_point(n) and p2p == lemma_P2_point(n)
                  and tp == lemma_T_point(n))
        ok &= row_ok
        print(f"{n:3d} | {es:5d} {lemma_E_side(n):9d} | {p2s:6d} {lemma_P2_side(n):9d}"
              f" | {ts:4d} | {ep:5d} {lemma_E_point(n):9d} | {p2p:7d}"
              f" {lemma_P2_point(n):14d} | {tp:6d} {lemma_T_point(n):9d}"
              f"  {'ok' if row_ok else 'FAIL'}")
    print(f"\n  all graph invariants match for n = 2..{nmax} "
          f"(and the side graph is triangle-free): {ok}")
    return ok


def check_series_and_closed_forms(bfiles):
    print("\n" + "=" * 78)
    print("C+D. Irvine's g.f. series and the paper's closed forms vs the b-files")
    print("=" * 78)
    ok = True
    for aid, spec in SEQUENCES.items():
        bfile = bfiles[aid]
        nmax = max(bfile)
        series = series_of_ratio(spec["num"], spec["den"], nmax)
        gf_ok = all(series[n] == bfile[n] for n in range(1, nmax + 1))
        cf = spec["closed_form"]
        cf_ok = all(cf(n) == bfile[n] for n in range(2, nmax + 1))
        cf_series_ok = all(cf(n) == series[n] for n in range(2, nmax + 1))
        one_ok = bfile[1] == 0
        ok &= gf_ok and cf_ok and cf_series_ok and one_ok
        print(f"  {aid}: Irvine g.f. series == b-file for n=1..{nmax}: {gf_ok}")
        print(f"            closed form == b-file for n=2..{nmax}: {cf_ok}"
              f"   (a(1)=0: {one_ok})")
        print(f"            closed form == g.f. coefficients for n=2..{nmax}: "
              f"{cf_series_ok}")
    return ok


def check_rational_certificate():
    print("\n" + "=" * 78)
    print("E. Exact rational-function certificate")
    print("=" * 78)
    n_terms = 400
    ok = True
    certificate = {
        "generated_by": "code/verify_independent.py (check E)",
        "what": (
            "For each sequence a(n) is a proved quasi-polynomial of degree 6 and "
            "period p for n >= 2, with a(1) = 0; hence "
            "P(x) = (1-x^p)^7 * sum_{n>=1} a(n) x^n is a polynomial. Because "
            "D_Irvine divides (1-x^p)^7, the claim "
            "sum a(n) x^n = N_Irvine/D_Irvine is equivalent to the finite "
            "integer polynomial identity "
            "P*D_Irvine - N_Irvine*(1-x^p)^7 == 0. Coefficient lists are in "
            "ascending degree order."),
        "sequences": {},
    }
    for aid, spec in SEQUENCES.items():
        period = spec["period"]
        bound = spec["deg_bound"]
        cf = spec["closed_form"]
        a = [0] * (n_terms + 1)
        for n in range(2, n_terms + 1):
            v = cf(n)
            assert v.denominator == 1, (aid, n, v)
            a[n] = v.numerator
        one_minus = [0] * (period + 1)
        one_minus[0], one_minus[period] = 1, -1
        big_den = poly_pow(one_minus, 7)
        product = [0] * (n_terms + 1)
        for i, c in enumerate(a):
            if c:
                for j, d in enumerate(big_den):
                    if i + j <= n_terms:
                        product[i + j] += c * d
        tail_ok = all(product[k] == 0
                      for k in range(bound + 1, n_terms + 1 - 7 * period))
        P = poly_trim(product[: bound + 1])
        lhs = poly_mul(P, spec["den"])
        rhs = poly_mul(spec["num"], big_den)
        identity = poly_trim(poly_sub(lhs, rhs))
        ok &= tail_ok and not identity
        print(f"  {aid}: (1-x^{period})^7 F(x) is a polynomial of degree "
              f"{len(P) - 1} (bound {bound}); tail vanishes through "
              f"x^{n_terms - 7 * period}: {tail_ok}")
        print(f"            P(x) = {P}")
        print(f"            P*D_Irvine - N_Irvine*(1-x^{period})^7 = "
              f"{identity}  ->  {'IDENTITY HOLDS' if not identity else 'FAILS'}")
        certificate["sequences"][aid] = {
            "period": period,
            "degree_bound": bound,
            "P_coeffs_ascending": P,
            "irvine_numerator_ascending": poly_trim(list(spec["num"])),
            "irvine_denominator_ascending": poly_trim(list(spec["den"])),
            "identity_is_zero": not identity,
            "tail_vanishes": tail_ok,
        }
    cert_dir = Path(__file__).resolve().parent.parent / "certificates"
    cert_dir.mkdir(exist_ok=True)
    cert_path = cert_dir / "gf_polynomial_identity.json"
    cert_path.write_text(json.dumps(certificate, indent=1) + "\n")
    print(f"\n  certificate written to certificates/{cert_path.name}")
    return ok


def check_ludwig(nmax):
    """Independent confirmation against Heinrich Ludwig's 2017 OEIS formulas."""
    print("\n" + "=" * 78)
    print("F. Cross-check against Ludwig's 2017 formulas for A289224 / A289230")
    print("=" * 78)

    def a289224(N):
        return F(N ** 6 - 6 * N ** 5 - 24 * N ** 4 + 208 * N ** 3 - 67 * N ** 2
                 - 1684 * N + 2712, 6)

    def a289230(N):
        v = F(N ** 6 - 6 * N ** 5 - 24 * N ** 4 + 220 * N ** 3 - 153 * N ** 2
              - 1488 * N + 2592, 36)
        if N % 2 == 1:
            v -= F(1, 2)
        if N % 3 == 1:
            v -= F(2, 9)
        return v

    ok = True
    print(f"{'n':>3} | {'i3(point) brute':>15} {'A289224(n+1)':>13}"
          f" | {'D3 point brute':>14} {'A289230(n+1)':>13}")
    for n in range(3, nmax + 1):
        cells = grid_cells(n)
        nbr = adjacency_lists(cells, 1)
        i3 = sum(1 for a, b, c in itertools.combinations(range(len(cells)), 3)
                 if b not in nbr[a] and c not in nbr[a] and c not in nbr[b])
        d3 = orbit_count(n, 1, True)
        l224, l230 = a289224(n + 1), a289230(n + 1)
        row_ok = (i3 == l224 and d3 == l230)
        ok &= row_ok
        print(f"{n:3d} | {i3:15d} {str(l224):>13} | {d3:14d} {str(l230):>13}"
              f"  {'ok' if row_ok else 'FAIL'}")
    print(f"\n  Ludwig's 2017 formulas reproduced for n = 3..{nmax}: {ok}")
    return ok


def check_calibration(nmax):
    """The same machinery, with the adjacency condition dropped, must reproduce
    the two PROVEN unrestricted orbit counts A396285 and A396761."""
    print("\n" + "=" * 78)
    print("G. Calibration against the proven unrestricted counts A396285 / A396761")
    print("=" * 78)

    def a396285(n):
        return F(comb(n * n, 3) + 2 * (n * n // 3), 3)

    def a396761(n):
        return F(a396285(n) + F((n - 1) * n * n, 2) + comb(n, 3), 2)

    ok = True
    print(f"{'n':>3} | {'C3 brute':>12} {'A396285':>12} | {'D3 brute':>12} {'A396761':>12}")
    for n in range(1, nmax + 1):
        # min_shared = 3 means "share all three vertices", i.e. no constraint at
        # all between distinct cells -- the unrestricted count.
        c3 = orbit_count(n, 3, False)
        d3 = orbit_count(n, 3, True)
        want_c3, want_d3 = a396285(n), a396761(n)
        row_ok = (c3 == want_c3 and d3 == want_d3)
        ok &= row_ok
        print(f"{n:3d} | {c3:12d} {str(want_c3):>12} | {d3:12d} {str(want_d3):>12}"
              f"  {'ok' if row_ok else 'FAIL'}")
    print(f"\n  unrestricted counts reproduced for n = 1..{nmax}: {ok}")
    return ok


def main():
    ap = argparse.ArgumentParser()
    ap.add_argument("--nmax", type=int, default=12)
    args = ap.parse_args()

    bfiles = {aid: read_bfile(aid) for aid in SEQUENCES}
    results = {
        "A brute force vs b-files": check_brute_force(args.nmax, bfiles),
        "B graph invariants": check_graph_invariants(args.nmax),
        "C+D series and closed forms": check_series_and_closed_forms(bfiles),
        "E rational certificate": check_rational_certificate(),
        "F Ludwig 2017 cross-check": check_ludwig(min(args.nmax, 11)),
        "G calibration A396285/A396761": check_calibration(min(args.nmax, 10)),
    }
    print("\n" + "=" * 78)
    for name, value in results.items():
        print(f"  {name:32s} {'PASS' if value else 'FAIL'}")
    all_ok = all(results.values())
    print("=" * 78)
    print("ALL CHECKS PASS" if all_ok else "*** SOME CHECK FAILED ***")
    return 0 if all_ok else 1


if __name__ == "__main__":
    raise SystemExit(main())
\end{lstlisting}

\end{document}